\documentclass[11pt, a4paper, leqno]{amsart}
\usepackage{amsfonts, amssymb, amsmath}
\usepackage{geometry}
\usepackage{tocvsec2}

\usepackage[T1]{fontenc}
\usepackage{lmodern}
\usepackage{microtype}

\usepackage{amsthm}
\usepackage{microtype}
\usepackage{enumerate}

\usepackage[abbrev, backrefs]{amsrefs}

\newtheorem{theorem}{Theorem}
\newtheorem{proposition}[theorem]{Proposition}
\newtheorem{lemma}[theorem]{Lemma}

\theoremstyle{definition}

\numberwithin{equation}{section}

\newcommand{\N}{{\mathbb N}}
\newcommand{\R}{{\mathbb R}}

\newcommand{\st}{\;\vert\;}

\newcommand{\abs}[1]{\lvert #1 \rvert}

\newcommand{\scalprod}[2]{( #1 \vert #2)}

\newcommand{\beq}{\begin{equation}}
\newcommand{\eeq}{\end{equation}}
\newcommand{\eeqa}{\begin{eqnarray}}
\newcommand{\beqa}{\end{eqnarray}}
\newcommand{\beqn}{\begin{eqnarray*}}
\newcommand{\eeqn}{\end{eqnarray*}}

\newcommand{\dif}{\,\mathrm{d}}

\title[Hardy-Littlewood-Sobolev critical exponent]
{Groundstates of nonlinear Choquard equations:\\ Hardy-Littlewood-Sobolev critical exponent}
\author{Vitaly Moroz}
\address{Swansea University\\ Department of Mathematics\\ Singleton Park\\
Swansea\\ SA2~8PP\\ Wales, United Kingdom}	
\email{V.Moroz@swansea.ac.uk}

\author{Jean Van Schaftingen}
\address{Universit\'e Catholique de Louvain\\ Institut de Recherche en Math\'ematique et Phy\-sique\\ Chemin du Cyclotron 2 bte L7.01.01\\ 1348 Louvain-la-Neuve \\ Belgium}
\email{Jean.VanSchaftingen@uclouvain.be}



\begin{document}

\begin{abstract}
We consider nonlinear Choquard equation
\begin{equation*}
 - \Delta u + V u  = \bigl(I_\alpha \ast \abs{u}^{\frac{\alpha}{N}+1}\bigr) \abs{u}^{\frac{\alpha}{N}-1} u\quad\text{in \(\R^N\)},
\end{equation*}
where \(N \ge 3\),  $V \in L^\infty (\R^N)$ is an external potential and \(I_\alpha (x)\)
is the Riesz potential of order \(\alpha \in (0, N)\).
The power $\frac{\alpha}{N}+1$ in the nonlocal part of the equation is critical with respect to the Hardy-Littlewood-Sobolev inequality.
As a consequence, in the associated minimization problem a loss of compactness may occur.
We prove that if $\liminf_{\abs{x} \to \infty} \bigl(1 - V (x)\bigr)\abs{x}^2 > \frac{N^2 (N - 2)}{4 (N + 1)}$
then the equation has a nontrivial solution. We also discuss some necessary conditions for the existence of a solution.
Our considerations are based on a concentration compactness argument and a nonlocal version of Brezis-Lieb lemma.
\end{abstract}

\maketitle

\maxtocdepth{section}
\tableofcontents

\section{Introduction and results}

We consider a nonlinear Choquard type equation
\begin{equation*}
\tag{$\mathcal{P}$}
\label{equationNLChoquard-p}
 - \Delta u + V u  = \bigl(I_\alpha \ast \abs{u}^p\bigr) \abs{u}^{p-2} u\quad\text{in \(\R^N\)},
\end{equation*}
where \(N \in \N\), $\alpha\in(0,N)$, $p>1$,
\(I_\alpha : \R^N\setminus\{0\} \to \R\) is the Riesz potential of order \(\alpha \in (0, N)\) defined for every \(x \in \R^N \setminus \{0\}\) by
\[
  I_\alpha(x)=\frac{\Gamma(\tfrac{N-\alpha}{2})}
                   {2^{\alpha}\pi^{N/2}\Gamma(\tfrac{\alpha}{2})\, \abs{x}^{N-\alpha}},
\]
and $V \in L^\infty (\R^N)$ is an external potential.

For \(N = 3\), \(\alpha = 2\) and \(p = 2\) equation \eqref{equationNLChoquard-p} is the  \emph{Choquard-Pekar equation}
which goes back to the 1954's work by S.\thinspace I.\thinspace Pekar  on quantum theory of a Polaron at rest \citelist{\cite{Pekar1954}\cite{DevreeseAlexandrov2009}*{Section 2.1}}
and to 1976's model of P.\thinspace Choquard of an electron trapped in its own hole, in an approximation to Hartree-Fock theory of one-component plasma \cite{Lieb1977}.
In the 1990's the same equation reemerged as a model of self-gravitating matter  \citelist{\cite{KRWJones1995newtonian}\cite{Moroz-Penrose-Tod-1998}}
and is known in that context as the \emph{Schr\"odinger-Newton equation}.

Mathematically, the existence and qualitative properties of solutions of Choquard equation \eqref{equationNLChoquard-p}
have been studied for a few decades by variational methods,
see \citelist{\cite{Lieb1977}\cite{Lions1980}\cite{Menzala1980}\cite{Lions1984-1}*{Chapter III}} for earlier and \citelist{\cite{Ma-Zhao-2010}\cite{CingolaniSecchiSquassina2010}\cite{CingolaniClappSecchi2012}\cite{CingolaniSecchi}\cite{ClappSalazar}\cite{MorozVanSchaftingen2013JFA}
\cite{MorozVanSchaftingenBL}\cite{MorozVanSchaftingen2014CalcVar}} for recent work on the problem and further references therein.

The following sharp characterisation of the existence and nonexistence of nontrivial solutions of \eqref{equationNLChoquard-p}
in the case of constant potential $V$ can be found in \cite{MorozVanSchaftingen2013JFA}.
\begin{theorem}
[Ground states of \eqref{equationNLChoquard-p} with constant potential \cite{MorozVanSchaftingen2013JFA}*{theorems 1 and 2}]
Assume that $V\equiv 1$.
Then \eqref{equationNLChoquard-p} has a nontrivial solution \(u \in H^1(\R^N)\cap L^\frac{2Np}{N-\alpha}(\R^N)\)
with $\nabla u\in H^1_{\mathrm{loc}}(\R^N)\cap L^\frac{2Np}{N-\alpha}_{\mathrm{loc}} (\R^N)$
if and only if \(p\in\big(\frac{\alpha}{N}+1,\frac{N + \alpha}{(N - 2)_+}\big)\).
\end{theorem}

If \(p\in\big[\frac{\alpha}{N}+1,\frac{N + \alpha}{(N - 2)_+}\big]\) then $H^1(\R^N)\subset L^\frac{2Np}{N-\alpha}(\R^N)$
by the Sobolev inequality,
and moreover, every $H^1$--solution of \eqref{equationNLChoquard-p} belongs to $W^{2,p}_{\mathrm{loc}}(\R^N)$ for any $p\ge 1$
by a regularity result in \cite{MorozVanSchaftingenBL}*{proposition~3.1}. This implies that the
Choquard equation \eqref{equationNLChoquard-p} with a positive constant potential has no $H^1$--solutions at the end-points
of the above existence interval.

In this note we are interested in the existence and nonexistence of solutions to \eqref{equationNLChoquard-p} with nonconstant
potential $V$ at the {\em lower critical exponent} $p=\frac{\alpha}{N}+1$, that is, we consider the problem
\begin{equation*}
\tag{$\mathcal{P_*}$}
\label{equationNLChoquard}
 - \Delta u + V u  = \bigl(I_\alpha \ast \abs{u}^{\frac{\alpha}{N}+1}\bigr) \abs{u}^{\frac{\alpha}{N}-1} u\quad\text{in \(\R^N\)}.
\end{equation*}
The exponent $\frac{\alpha}{N}+1$ is critical with respect
to the Hardy-Littlewood-Sobolev inequality, which we recall here in a form of minimization problem
\[
  c_\infty =  \inf \Bigl\{ \int_{\R^N} \abs{u}^2 \st u \in L^2 (\R^N) \text{ and }\int_{\R^N} \bigl(I_\alpha \ast \abs{u}^{\frac{\alpha}{N} + 1}\bigr) \abs{u}^{\frac{\alpha}{N} + 1} = 1\Bigr\} > 0.
\]

\begin{theorem}[Optimal Hardy-Littlewood-Sobolev inequality \citelist{\cite{Lieb1983}*{theorem 3.1}\cite{LiebLoss2001}*{theorem 4.3}}]
\label{theoremHLS}
The infimum \(c_\infty\) is achieved if and only if
\begin{equation}\label{HLSoptimal}
  u(x) = C\left(\frac{\lambda}{\lambda^2 + \abs{x - a}^2}\right)^{N/2},
\end{equation}
where \(C > 0\) is a fixed constant, \(a \in \R^N\) and \(\lambda \in (0, \infty)\) are parameters.
\end{theorem}

The form of minimizers in theorem \ref{theoremHLS} suggests that a loss of compactness in \eqref{equationNLChoquard}
may occur by translations and dilations.

In order to characterise the existence of nontrivial solutions for the lower critical Choquard equation \eqref{equationNLChoquard}
we define the critical level
\[
  c_* = \inf \Bigl\{ \int_{\R^N} \abs{\nabla u}^2 + V \abs{u}^2 \st u \in H^1 (\R^N) \text{ and } \int_{\R^N} \bigl(I_\alpha \ast \abs{u}^{\frac{\alpha}{N} + 1}\bigr) \abs{u}^{\frac{\alpha}{N} + 1} = 1\Bigr\}.
\]
It can be checked directly that if $u\in H^1(\R^N)$ achieves the infimum \(c_*\), then a multiple of the minimizer $u$ is a weak solution of Choquard equation \eqref{equationNLChoquard}.
\smallskip

Using a Brezis-Lieb type lemma for Riesz potentials \cite{MorozVanSchaftingen2013JFA}*{lemma 2.4}
and a concentration compactness argument (lemma \ref{lemmaWeakHilbert}), we establish our main abstract result.

\begin{theorem}[Existence of a minimizer]\label{theoremExistence}
Assume that \(V \in L^\infty (\R^N)\) and
\begin{equation}\label{eqliminfV}
\liminf_{\abs{x} \to \infty} V (x) \ge 1.
\end{equation}
If \(c_* < c_\infty\) then the infimum \(c_*\) is achieved and every minimizing sequence for $c_*$
up to a subsequence converges strongly in \(H^1(\R^N)\).
\end{theorem}

The inequality for the existence of minimizers is sharp, as shown by the following lemma for constant potentials.
\begin{lemma}\label{lemmaV1Sharp}
If $V\equiv 1$, then $c_*=c_\infty$.
\end{lemma}

Since problem \eqref{equationNLChoquard} with $V\equiv 1$ has no $H^1$--solutions,
this shows that the strict inequality $c_*<c_\infty$ is indeed essential
for the existence of a minimizer for $c_*$.
\smallskip

In fact, the strict inequality $c_*<c_\infty$ is necessary
at least for the strong convergence of {\em all} minimizing sequences.

\begin{proposition}\label{propositionSharp}
Let \(V \in L^\infty (\R^N)\). If
\[
 \limsup_{\abs{x} \to \infty} V (x) \le 1,
\]
then
\[
  c_* \le c_\infty.
\]
In addition, if
\[
 c_*=c_\infty,
\]
then there exists a minimizing sequence for $c_*$ which converges weakly to \(0\) in \(H^1(\R^N)\).
\end{proposition}

Using Hardy-Littlewood-Sobolev minimizers \eqref{HLSoptimal} as a family of test functions for $c_*$,
we establish a sufficient condition for the strict inequality $c_*<c_\infty$.

\begin{theorem}\label{theoremSufficient-Hardy}
Let \(V \in L^\infty (\R^N)\). If
\begin{equation}\label{eq:sufficient-Hardy}
  \liminf_{\abs{x} \to \infty} \bigl(1 - V (x)\bigr)\abs{x}^2 > \frac{N^2 (N - 2)_+}{4 (N + 1)},
\end{equation}
then $c_*<c_\infty$ and hence the infimum \(c_*\) is achieved.
\end{theorem}

In particular, if \(N=1,2\) then condition \eqref{eq:sufficient-Hardy} reduces to
\[
  \liminf_{\abs{x} \to \infty} \bigl(1 - V (x)\bigr)\abs{x}^2 > 0,
\]
that is, the potential \(1-V\) should not decay to zero at infinity
faster then the inverse square of $|x|$.

Employing a version of Poho\v zaev identity for Choquard equation \eqref{equationNLChoquard}
(see proposition \ref{propositionPohozaev} below),
we show that a certain control on the potential $V$ is indeed necessary
for the strict inequality $c_*<c_\infty$.

\begin{proposition}\label{propositionPohozaev1}
Let \(V \in C^1 (\R^N) \cap L^\infty (\R^N)\).
If
\begin{equation}\label{equationSufficient1}
  \sup \Bigl\{\int_{\R^N} \frac{1}{2}\scalprod{\nabla V (x)}{x} \abs{\varphi (x)}^2\dif x \st \varphi \in C^1_c (\R^N) \text{ and } \int_{\R^N} \abs{\nabla \varphi}^2 \le 1\Bigr\} < 1,
\end{equation}
then Choquard equation \eqref{equationNLChoquard}
does not have a nonzero solution $u\in H^1(\R^N)\cap W^{2, 2}_\mathrm{loc}(\R^N)$.
\end{proposition}

In particular, combining \eqref{equationSufficient1} with Hardy's inequality on $\R^N$,
we obtain a simple nonexistence criterion.

\begin{proposition}\label{propositionPohozaev2}
Let $N\ge 3$ and \(V \in C^1 (\R^N) \cap L^\infty (\R^N)\).
If for every \(x \in \R^N\),
\begin{equation}\label{eq:PohozaevHardy}
  \sup_{x \in \R^N} \abs{x}^2 \scalprod{\nabla V (x)}{x} < \frac{(N - 2)^2}{2},
\end{equation}
then Choquard equation \eqref{equationNLChoquard}
does not have a nonzero solution $u\in H^1(\R^N)\cap W^{2, 2}_\mathrm{loc}(\R^N)$.
\end{proposition}

For example, for $N\ge 3$ and $\mu>0$, we consider a model equation
\begin{equation}\label{eq:model}
 - \Delta u + \Big(1- \frac{\mu}{1+\abs{x}^2}\Big) u  = \bigl(I_\alpha \ast \abs{u}^{\frac{\alpha}{N}+1}\bigr) \abs{u}^{\frac{\alpha}{N}-1} u\quad\text{in \(\R^N\)}.
\end{equation}
Then proposition~\ref{propositionPohozaev2} implies that \eqref{eq:model} has no nontrivial solutions
for $\mu < \frac{(N-2)^2}{4}$, while for $\mu >\frac{N^2 (N - 2)}{4 (N + 1)}$ assumption \eqref{eq:sufficient-Hardy} is satisfied and hence \eqref{equationNLChoquard} admits a groundstate.
We note that
\[
  \frac{\frac{(N - 2)^2}{4}}{\frac{N^2 (N - 2)}{4 (N + 1)}} = 1 - \frac{N - 2}{N^2},
\]
so that the two bounds are asymptotically sharp when \(N \to \infty\).
We leave as an open question whether \eqref{eq:model} admits a ground state for
$\mu \in \big[\frac{(N-2)^2}{4},\frac{N^2 (N - 2)}{4 (N + 1)}\big]$.
\smallskip

We emphasise that unlike the asymptotic sufficient existence condition \eqref{eq:sufficient-Hardy},
nonexistence condition \eqref{eq:PohozaevHardy}
is a global condition on the whole of $\R^N$.
For example, a direct computation shows that
for $a=0$ and every $\lambda>0$,
a multiple of the Hardy-Littlewood-Sobolev minimizer \eqref{HLSoptimal} solves the equation
\begin{equation}\label{eq:null}
 - \Delta u + \Big(1+ \frac{N(2\abs{x}^2-N\lambda^2)}{(\abs{x}^2+\lambda^2)^2}\Big) u  = \bigl(I_\alpha \ast \abs{u}^{\frac{\alpha}{N}+1}\bigr) \abs{u}^{\frac{\alpha}{N}-1} u\quad\text{in \(\R^N\)}.
\end{equation}
Here \eqref{eq:PohozaevHardy} fails on an annulus centered at the origin,
while $V(x)>1$ and $\scalprod{\nabla V (x)}{x}<0$ for all $\abs{x}$ sufficiently large.
Moreover,
\[
 \lim_{\abs{x} \to \infty} (1 - V (x)) \abs{x}^2 = - 2 N < 0 \le \frac{N^2 (N - 2)_+}{4 (N + 1)}.
\]
Note that the constructed solution $u_\lambda$ satisfies
$$\int_{\R^N} \abs{\nabla u_\lambda}^2 + V \abs{u_\lambda}^2=0.$$
In particular, we are unable to conclude that $c_*<c_\infty$.
We do not know whether \(u_\lambda\) is a groundstate of \eqref{eq:null}. However, if $u_\lambda$ was not a groundstate, then we would have \(c_* < c_\infty\) and \eqref{eq:null} would then have a groundstate
by theorem~\ref{theoremExistence}.

\section{Existence of minimizers under strict inequality: proof of theorem \ref{theoremExistence}}

In order to prove theorem \ref{theoremExistence} we will use a special case
of the classical Brezis-Lieb lemma \cite{BrezisLieb1983} for Riesz potentials.

\begin{lemma}[Brezis-Lieb lemma for the Riesz potential \cite{MorozVanSchaftingen2013JFA}*{lemma 2.4}]
\label{lemmaBrezisLiebriesz}
Let \(N \in \N\), \(\alpha \in (0, N)\),
and \((u_n)_{n \in \N}\) be a bounded sequence in \(L^2 (\R^N)\).
If \(u_n \to u\) almost everywhere on \(\R^N\) as \(n \to \infty\), then
\begin{multline*}
  \int_{\R^N} \bigl(I_\alpha \ast \abs{u}^{\frac{\alpha}{N} + 1}\bigr) \abs{u}^{\frac{\alpha}{N} + 1} = \lim_{n \to \infty} \int_{\R^N} \bigl(I_\alpha \ast \abs{u_n}^{\frac{\alpha}{N} + 1}\bigr) \abs{u_n}^{\frac{\alpha}{N} + 1}\\
 - \int_{\R^N} \bigl(I_\alpha \ast \abs{u_n - u}^{\frac{\alpha}{N} + 1}\bigr) \abs{u_n - u}^{\frac{\alpha}{N} + 1}.
\end{multline*}
\end{lemma}

Our second result is a concentration type lemma.

\begin{lemma}
\label{lemmaWeakHilbert}
Assume that \(V \in L^\infty (\R^N)\) and \(\liminf_{\abs{x} \to \infty} V (x) \ge 1\).
If the sequence \((u_n)_{n \in \N}\) is bounded in \(L^2 (\R^N)\) and converges in \(L^2_\mathrm{loc} (\R^N)\) to \(u\) as \(n \to \infty\),
then
\[
  \int_{\R^N} V \abs{u}^2
  \le \liminf_{n \to \infty} \int_{\R^N} V \abs{u_n}^2 - \int_{\R^N} \abs{u_n - u}^2.
\]
\end{lemma}

\begin{proof}
Since the sequence \((u_n)_{n \in \N}\) is bounded in \(L^2 (\R^N)\) and converges in measure to \(u\), we deduce by the Brezis-Lieb lemma
\cite{BrezisLieb1983} (see also \cite{LiebLoss2001}*{theorem 1.9}) that
\[
  \int_{\R^N}  V \abs{u}^2
  = \lim_{n \to \infty} \int_{\R^N} V \abs{u_n}^2 - \int_{\R^N}  V \abs{u_n - u}^2.
\]
Now, we observe that for every \(R > 0\) and every \(n \in \N\),
\[
  \int_{\R^N} (1 - V) \abs{u_n - u}^2
  \le \int_{B_R} (1 - V) \abs{u_n - u}^2 + (1 - \inf_{\R^N \setminus B_R} V)_+ \int_{\R^N} \abs{u_n - u}^2.
\]
By the local \(L^2_\mathrm{loc} (\R^N)\) convergence, we note that
\[
  \lim_{n \to \infty} \int_{B_R} (1 - V) \abs{u_n - u}^2 = 0.
\]
Since \(\lim_{R \to \infty} (1 - \inf_{\R^N \setminus B_R} v)_+ = 0\) and \((u_n - u)_{n \in \N}\) is bounded in \(L^2 (\R^N)\), we conclude that
\[
  \limsup_{n \to \infty} \int_{\R^N} (1 - V) \abs{u_n - u}^2 \le 0;
\]
the conclusion follows.
\end{proof}

\begin{proof}[Proof of theorem~\ref{theoremExistence}]
Let \((u_n)_{n \in \N}\subset H^1 (\R^N)\) be a minimizing sequence for \(c_*\), that is
\[
  \int_{\R^N} \bigl(I_\alpha \ast \abs{u_n}^{\frac{\alpha}{N} + 1}\bigr)\abs{u_n}^{\frac{\alpha}{N} + 1} = 1
\]
and
\[
  \lim_{n \to \infty} \int_{\R^N} \abs{\nabla u_n}^2 + V \abs{u_n}^2 \to c_*.
\]
In view of our assumption \eqref{eqliminfV} we observe that the sequence $(u_n)_{n \in \N}$ is bounded in $H^1 (\R^N)$.
So, there exists \(u \in H^1 (\R^N)\) such that, up to a subsequence, the sequence \((u_n)_{n \in \N}\) converges to \(u\) weakly in \(H^1 (\R^N)\) and, by the classical Rellich-Kondrachov compactness theorem, strongly in \(L^2_\mathrm{loc} (\R^N)\).
By the lower semi-continuity of the norm under weak convergence,
\[
  \int_{\R^N} \abs{\nabla u}^2 + V \abs{u}^2\le \lim_{n \to \infty} \int_{\R^N} \abs{\nabla u_n}^2 + V \abs{u_n}^2 = c_*.
\]
and by Fatou's lemma
\[
  \int_{\R^N} \bigl(I_\alpha \ast \abs{u_n}^{\frac{\alpha}{N} + 1}\bigr)\abs{u_n}^{\frac{\alpha}{N} + 1} \le 1.
\]
In order to conclude, it suffices to prove that equality is achieved in the latter inequality.

We observe that by lemma~\ref{lemmaBrezisLiebriesz},
\[
  \lim_{n \to \infty} \int_{\R^N} \bigl(I_\alpha \ast \abs{u_n - u}^{\frac{\alpha}{N} + 1}\bigr)\abs{u_n - u}^{\frac{\alpha}{N} + 1}
  = 1 - \int_{\R^N} \bigl(I_\alpha \ast \abs{u}^{\frac{\alpha}{N} + 1}\bigr)\abs{u}^{\frac{\alpha}{N} + 1}
\]
while by lemma~\ref{lemmaWeakHilbert} and by the lower-semicontinuity of the norm under weak convergence,
\begin{equation}
\begin{split}\label{eqBL-2}
\int_{\R^N} \abs{\nabla u}^2 + V \abs{u}^2
&\le\liminf_{n \to \infty} \int_{\R^N} \abs{\nabla u_n}^2 +
  \liminf_{n \to \infty} \int_{\R^N} V \abs{u_n}^2 - \abs{u_n - u}^2\\
  &\le \liminf_{n \to \infty} \int_{\R^N} \abs{\nabla u_n}^2 + V \abs{u_n}^2 - \abs{u_n - u}^2\\
  &= c_* - \limsup_{n \to \infty} \int_{\R^N}\abs{u_n - u}^2.
\end{split}
\end{equation}
By definition of \(c_\infty\), we have
\[
   \int_{\R^N} \abs{u_n - u}^2
   \ge c_\infty \Bigl(\int_{\R^N} \bigl(I_\alpha \ast \abs{u_n - u}^{\frac{\alpha}{N} + 1}\bigr)\abs{u_n - u}^{\frac{\alpha}{N} + 1} \Bigr)^\frac{N}{N + \alpha}.
\]
Therefore, we conclude that
\[
  \int_{\R^N} \abs{\nabla u}^2 + V \abs{u}^2
  \le c_* - c_\infty\Bigl( 1 - \int_{\R^N} \bigl(I_\alpha \ast \abs{u}^{\frac{\alpha}{N} + 1}\bigr)\abs{u}^{\frac{\alpha}{N} + 1}\Bigr)^\frac{N}{N + \alpha}.
\]
In view of the definition of \(c_*\) this implies that
\[
  c_* \ge  c_\infty \Bigl( 1 - \int_{\R^N} \bigl(I_\alpha \ast \abs{u}^{\frac{\alpha}{N} + 1}\bigr)\abs{u}^{\frac{\alpha}{N} + 1}\Bigr)^\frac{N}{N + \alpha} + c_* \Bigl( \int_{\R^N} \bigl(I_\alpha \ast \abs{u}^{\frac{\alpha}{N} + 1}\bigr)\abs{u}^{\frac{\alpha}{N} + 1}\Bigr)^\frac{N}{N + \alpha}.
\]
Since by assumption \(c_* < c_\infty\), we conclude that
\[
  \int_{\R^N} \bigl(I_\alpha \ast \abs{u}^{\frac{\alpha}{N} + 1}\bigr)\abs{u}^{\frac{\alpha}{N} + 1}=1,
\]
and hence, by definition of \(c_*\),
\[
  \int_{\R^N} \abs{\nabla u}^2 + V \abs{u}^2 = c_*,
\]
that is the infimum $c_*$ is achieved at $u$. Moreover, from \eqref{eqBL-2} we conclude that $u_n\to u$ in $L^2(\R^N)$.
Since $V\in L^\infty(\R^N)$, this implies that $Vu_n\to Vu$ in $L^2(\R^N)$.
Using \eqref{eqBL-2} again, we conclude that
\[
  \int_{\R^N} \abs{\nabla u}^2=\lim_{n \to \infty} \int_{\R^N} \abs{\nabla u_n}^2.
\]
Since \((u_n)_{n \in \N}\) converges to \(u\) weakly in \(H^1 (\R^N)\),
this implies that \((u_n)_{n \in \N}\) also converges to \(u\) strongly in \(H^1 (\R^N)\).
\end{proof}

\section{Optimality of the strict inequality}

In this section we prove lemma \ref{lemmaV1Sharp} and proposition \ref{propositionSharp}.

\begin{proof}[Proof of lemma \ref{lemmaV1Sharp}]
Let us denote by \(\Tilde{c}_\infty\) the infimum on the right-hand side. By density of the space \(H^1 (\R^N)\) in \(L^2 (\R^N)\) and by continuity in \(L^2\) of the integral functionals involved in the definition of \(c_\infty\), it is clear that \(\Tilde{c}_\infty \ge c_\infty\). We choose now \(u \in H^1 (\R^N)\) and define for \(\lambda > 0\) the function \(u_\lambda \in H^1 (\R^N)\)  for every \(x \in \R^N\) by
$$u_\lambda (x) = \lambda^{N/2}u (\lambda x).$$
We compute for every \(\lambda > 0\) that
\[
  \int_{\R^N} \bigl(I_\alpha \ast \abs{u_\lambda}^{\frac{\alpha}{N} + 1}\bigr) \abs{u_\lambda}^{\frac{\alpha}{N} + 1}=\int_{\R^N} \bigl(I_\alpha \ast \abs{u}^{\frac{\alpha}{N} + 1}\bigr) \abs{u}^{\frac{\alpha}{N} + 1}
\]
and
\[
  \int_{\R^N} \abs{\nabla u_\lambda}^2 + \abs{u_\lambda}^2 = \frac{1}{\lambda^2} \int_{\R^N} \abs{\nabla u}^2 + \int_{\R^N} \abs{u}^2.
\]
Hence,
\[
  \inf_{\lambda > 0} \int_{\R^N} \abs{\nabla u_\lambda}^2 + \abs{u_\lambda}^2 = \int_{\R^N} \abs{u}^2,
\]
and we conclude that \(\Tilde{c}_\infty \le c_\infty\).
\end{proof}

\begin{proof}[Proof of proposition \ref{propositionSharp}]
For $\lambda>0$, let
$$
 u_\lambda (x) =  C\left(\frac{\lambda}{\lambda^2 + \abs{x}^2}\right)^\frac{N}{2} = \lambda^{-\frac{N}{2}}u_1\Big(\frac{x}{\lambda}\Big)
$$
be a family of minimizers for $c_\infty$ given in \eqref{HLSoptimal}.
We observe that
\[
 \int_{\R^N} \bigl(I_\alpha \ast \abs{u_\lambda}^{\frac{\alpha}{N} + 1}\bigr) \abs{u_\lambda}^{\frac{\alpha}{N} + 1}= 1,
\]
whereas by a change of variables,
\[
 \int_{\R^N} \abs{\nabla u_\lambda}^2 + V \abs{u_\lambda}^2
 = \frac{1}{\lambda^2} \int_{\R^N} \abs{\nabla u_1}^2
  + \int_{\R^N} V \Big(\frac{y}{\lambda}\Big) \frac{C^2}{1 + \abs{y}^2}\dif y.
\]
By Lebesgue's dominated convergence theorem
\[
 \limsup_{\lambda \to 0} \int_{\R^N} V \Big(\frac{y}{\lambda}\Big) \frac{C^2}{1 + \abs{y}^2}\dif y
 \le \int_{\R^N} \frac{C^2}{1 + \abs{y}^2}\dif y = c_\infty,
\]
so we conclude that $c_*\le c_\infty$. If, in addition, $c_*=c_\infty$ then for any $\lambda_n\to 0$,
$(u_{\lambda_n})_{n\in\N}$ is a minimizing sequence for $c_*$, and the conclusion follows.
\end{proof}

\section{Sufficient conditions for the strict inequality: proof of theorem \ref{theoremSufficient-Hardy}}

For $a\in\R^N$ and $\lambda>0$, let
\[
 u_\lambda (x) =  C\left(\frac{\lambda}{\lambda^2 + \abs{x-a}^2}\right)^{N/2}
\]
be a family of minimizers for $c_\infty$ as in \eqref{HLSoptimal}.
Then
\[
  \int_{\R^N} \abs{\nabla u_\lambda}^2 + V \abs{u_\lambda}^2
  = c_\infty + \int_{\R^N} \abs{\nabla u_\lambda}^2 + \int_{\R^N} (V - 1) \abs{u_\lambda}^2.
\]
Denote
\[
\mathcal I_V(a,\lambda):=\lambda^2\int_{\R^N} \abs{\nabla u_\lambda}^2 + \lambda^2\int_{\R^N} (V - 1) \abs{u_\lambda}^2 < 0.
\]
To obtain a sufficient conditions for $c_*<c_\infty$ it is enough to show that for some $a\in\R^N$,
\begin{equation}\label{eq:presufficient}
\inf_{\lambda>0}\mathcal{I}_V (a, \lambda)<0,
\end{equation}

\begin{proof}[Proof of theorem \ref{theoremSufficient-Hardy}.]
If \(N \le 2\), then  by \eqref{eq:sufficient-Hardy} there exists \(\mu > 0\) such that
\[
  \liminf_{\abs{x} \to \infty} (1 - V (x))\abs{x}^2 \ge \mu.
\]
Therefore
\[
  \lim_{\lambda \to \infty} \lambda^2 \int_{\R^N} (1 - V) \abs{u_\lambda}^2
  = \lim_{\lambda \to \infty} \int_{\R^N} \frac{\lambda^2(1 - V (\lambda x))}{(1 + \abs{x}^2)^N} \dif x \ge \int_{\R^N} \frac{\mu}{\abs{x}^2 (1 + \abs{x}^2)^N} \dif x = \infty.
\]
Since for every \(\lambda > 0\),
\[
  \lambda^2 \int_{\R^N} \abs{\nabla u_\lambda}^2 = \int_{\R^N} \abs{\nabla u_1}^2 < \infty,
\]
the condition \eqref{eq:presufficient} is satisfied.

If $N\ge 3$, we observe that for every \(\lambda > 0\),
\[
  \int_{\R^N} \abs{\nabla u_\lambda}^2 = \frac{N^2 (N - 2)}{4 (N + 1)} \int_{\R^N} \frac{\abs{u_\lambda (x)}^2}{\abs{x}^2}\dif x.
\]
This follows from the fact that
\[
  \int_{\R^N} \frac{\abs{x}^2}{(1 + \abs{x}^2)^{N + 2}}\dif x = \frac{N - 2}{4 (N + 1)} \int_{\R^N} \frac{1}{\abs{x}^2 (1 + \abs{x}^2)^{N}}\dif x,
\]
which can be proved by two successive integrations by parts.
Then, after a transformation $x=\lambda y+a$,
\begin{equation*}
\mathcal I_V(a,\lambda)=
\int_{\R^N} \Big(\frac{\tfrac{N^2 (N - 2)}{4 (N + 1)}}{\abs{y}^2} - \lambda^2 (1 - V (a+\lambda y)) \Big)\frac{C^2}{(1 + \abs{y}^2)^N} \dif y,
\end{equation*}
and in view of \eqref{eq:sufficient-Hardy}, sufficient condition is \eqref{eq:presufficient} is satisfied for $a=0$, so we conclude that $c_*<c_\infty$.
\end{proof}

Note that if the function \(\lambda \mapsto \lambda^2(1 - V (a+\lambda y))\) is nondecreasing
for every $y\in\R^N$, then \(\lambda \mapsto \mathcal{I}_V (a, \lambda)\) is nonincreasing. Therefore
\(\mathcal{I}_V (a, \lambda)\) admits negative values if and only if it has a negative limit as $\lambda\to\infty$.
The latter is ensured in theorem~\ref{theoremSufficient-Hardy} via asymptotic condition \eqref{eq:sufficient-Hardy}.
This explains that if the function \(\lambda \mapsto \lambda^2(1 - V (a+\lambda y))\) is nondecreasing,
like for instance, in the special case
$$V(x)=1- \frac{\mu}{1+\abs{x}^2},$$
then integral sufficient condition \eqref{eq:presufficient} is in fact equivalent
to the asymptotic sufficient condition \eqref{eq:sufficient-Hardy}.

\section{Poho\v zaev identity and necessary conditions for the existence}

We establish a Poho\v{z}aev type identity,
which extends the identities \eqref{eqPohozhaev} obtained previously for constant potentials \(V\) \citelist{\cite{Menzala1983}\cite{CingolaniSecchiSquassina2010}*{lemma 2.1}\cite{MorozVanSchaftingen2013JFA}*{proposition 3.1}\cite{MorozVanSchaftingenBL}*{theorem 3}}.

\begin{proposition}
\label{propositionPohozaev}
Let $N\ge 3$ and \(V \in C^1 (\R^N) \cap L^\infty (\R^N)\) and \(u \in W^{1, 2} (\R^N)\). If
\[
  \sup_{x \in \R^N} \abs{\scalprod{\nabla V (x)}{}x} < \infty,
\]
and \(u \in W^{2, 2}_\mathrm{loc}(\R^N)\) satisfies Choquard equation \eqref{equationNLChoquard} then
\[
  \int_{\R^N} \abs{\nabla u}^2 = \frac{1}{2} \int_{\R^N} \scalprod{\nabla V (x)}{x}\, \abs{u (x)}^2\dif x.
\]
\end{proposition}
\begin{proof}
We fix a cut-off function \(\varphi \in C^1_c(\R^N)\) such that \(\varphi = 1\) on \(B_1\) and  we test for \(\lambda \in (0, \infty)\) the equation against the function \(v_\lambda \in W^{1, 2} (\R^N)\) defined for every \(x \in \R^N\) by
\[
 v_\lambda (x) = \varphi(\lambda x) \scalprod{\nabla u (x)}{x}
\]
to obtain the identity
\[
   \int_{\R^N} \scalprod{\nabla u}{\nabla v_\lambda} + \int_{\R^N} V u v_\lambda
= \int_{\R^N} (I_\alpha \ast \abs{u}^{\frac{\alpha}{N} + 1}) \abs{u}^{\frac{\alpha}{N} - 1} u v_\lambda.
\]
We compute for every \(\lambda > 0\), by definition of \(v_\lambda\), the chain rule and by the Gauss integral formula,
\[
\begin{split}
 \int_{\R^N} V u v_\lambda &= \int_{\R^N} V (x) u(x) \varphi (\lambda x)\scalprod{x}{\nabla u(x)} \dif x\\
  &= \int_{\R^N} V (x) \varphi (\lambda x)\scalprod{x}{\nabla \bigl(\tfrac{\abs{u}^2}{2}\bigr) (x)} \dif x\\
  &= - \int_{\R^N}  \bigl( (N V (x) + \scalprod{\nabla V (x)}{x}) \varphi (\lambda x)
                         + V (x) \lambda \scalprod{x}{\nabla \varphi (\lambda x)} \bigr)
                    \frac{\abs{u(x)}^2}{2}
        \dif x.
\end{split}
\]
Since \(\sup_{x \in \R^N} \scalprod{\nabla V (x)}{x} < \infty\),
by Lebesgue's dominated convergence theorem it holds
\[
 \lim_{\lambda \to 0} \int_{\R^N} V u v_\lambda = - \frac{N}{2} \int_{\R^N} V \abs{u}^2 - \frac{1}{2} \int_{\R^N} \scalprod{\nabla V (x)}{x} \abs{u}^2.
\]
By Lebesgue's dominated convergence again, since \(u \in W^{1, 2} (\R^N)\), we have (see \cite{MorozVanSchaftingen2013JFA}*{proof of proposition 3.1} for the details)
\[
 \lim_{\lambda \to 0} \int_{\R^N} \scalprod{\nabla u}{\nabla v_\lambda} = - \frac{N - 2}{2} \int_{\R^N} \abs{\nabla u}^2.
\]
and
\[
 \lim_{\lambda \to 0} \int_{\R^N} (I_\alpha \ast \abs{u}^{\frac{\alpha}{N} + 1}) \abs{u}^{\frac{\alpha}{N} - 1} u\, v_\lambda
= - \frac{N}{2} \int_{\R^N} (I_\alpha \ast \abs{u}^{\frac{\alpha}{N} + 1}) \abs{u}^{\frac{\alpha}{N} + 1}.
\]
We have thus proved the Poho\v{z}aev type identity
\begin{multline}
\label{eqPohozhaev}
  \frac{N - 2}{2} \int_{\R^N} \abs{\nabla u}^2 + \frac{N}{2} \int_{\R^N} V \abs{u}^2
 + \frac{1}{2} \int_{\R^N} \scalprod{\nabla V (x)}{x} \, \abs{u (x)}^2 \dif x\\ = \frac{N}{2} \int_{\R^N} (I_\alpha \ast \abs{u}^{\frac{\alpha}{N} + 1}) \abs{u}^{\frac{\alpha}{N} + 1}.
\end{multline}
If we test the equation against \(u\), we obtain the identity
\[
  \int_{\R^N} \abs{\nabla u}^2 + \int_{\R^N} V \abs{u}^2 = \int_{\R^N} (I_\alpha \ast \abs{u}^{\frac{\alpha}{N} + 1}) \abs{u}^{\frac{\alpha}{N} + 1};
\]
the combination of those two identities yields the conclusion.
\end{proof}

\begin{proof}[Proof of propositions \ref{propositionPohozaev1} and \ref{propositionPohozaev2}.]
Proposition \ref{propositionPohozaev1} is a direct consequence of proposition \ref{propositionPohozaev},
while proposition \ref{propositionPohozaev2} follows from proposition \ref{propositionPohozaev}
and the classical optimal Hardy inequality on $\R^N$,
\[
  \frac{(N - 2)^2}{4}\int_{\R^N} \frac{\abs{u(x)}^2}{\abs{x}^2}\dif x
  \le \int_{\R^N} \abs{\nabla u}^2
\]
which is valid for all $u\in H^1(\R^N)$ (see for example \cite{Willem2013}*{theorem 6.4.10 and exercise 6.8}).
\end{proof}

\end{document}